\documentclass[11pt]{amsart}   
\usepackage{amsmath, amssymb}  
\usepackage{amsthm}

\usepackage{verbatim}
\usepackage{amsfonts}
\usepackage{amssymb}
\usepackage{mathrsfs}
\usepackage[dvips]{graphicx}
\usepackage{pstricks}

\usepackage{setspace}
\usepackage[inner=2.4cm,outer=4cm,top=2.4cm,bottom=2.4cm,marginparwidth=90pt]{geometry}   
\usepackage{times}

\usepackage[notref,notcite,final]{showkeys}  
  
\numberwithin{equation}{section}   
 
\usepackage{citesort}  
 
\usepackage{framed}  

\newcommand{\dx}{\,dx}

\newcommand{\bigR}{{\mathbb R}}

\newtheorem{theorem}{Theorem}[section]  
\newtheorem{lemma}[theorem]{Lemma}  
\newtheorem{proposition}[theorem]{Proposition}  
\newtheorem{corollary}[theorem]{Corollary}  
\newtheorem{remark}[theorem]{Remark}   
 
\newtheorem{assumption}[theorem]{Assumption}

\begin{document}  
  
\title{On the convergence of axially symmetric volume preserving mean curvature flow}  
\author{Maria Athanassenas}  
\address{Maria Athanassenas,   School of Mathematical Sciences,  
Monash University,  
Vic 3800  
Australia}  
\email{maria.athanassenas@monash.edu} 
  
\author{Sevvandi Kandanaarachchi}  
\address{Sevvandi Kandanaarachchi, School of Mathematical Sciences,  
Monash University,  
Vic 3800  
Australia}  
\email{sevvandi.kandanaarachchi@monash.edu} 

 \onehalfspacing  
  

 %------------------------------ Introduction ------------------------------------------------

\begin{abstract}
We study the convergence of an axially symmetric hypersurface evolving by volume-preserving mean curvature flow. Assuming the surface is not pinching off along the axis at any time during the flow, and without any additional conditions, as for example on the curvature, 
we prove that it converges to a hemisphere, when the hypersurface has a free boundary and satisfies Neumann boundary data, and to a sphere when it is compact without boundary. 
\end{abstract}

\maketitle

\section{Introduction}
Consider $n$-dimensional hypersurfaces $M_t$, defined by a one parameter family of 
smooth immersions $\mathbf{x}_t: M^n \rightarrow \mathbb{R}^{n+1}$. The hypersurfaces
$M_t$ are said to move by mean curvature, if $\mathbf{x}_t = \mathbf{x}(\cdot, t)$ 
satisfies
\begin{equation}\label{eq:int_2}
\frac{d}{dt} \mathbf{x}(p, t)=-H(p,t)\nu(p,t), \hspace{5 mm} p\in M^n, t>0 \, .
\end{equation}
By $\nu(p,t)$ we denote a designated outer unit normal of $M_t$ at $\mathbf{x}(p,t)$ (outer normal in case of compact surfaces without boundary), and
by $H(p,t)$ the mean curvature with respect to this normal.

If the evolving compact surfaces $M_t$ are assumed to enclose a prescribed volume $V$ the evolution equation
changes as follows: 
\begin{equation}\label{eq:int_1}
\frac{d}{dt} \mathbf{x}(p,t) = -\left( H(p,t) -h(t)\right)\nu(p,t), 
\hspace{5 mm} p\in M^n, t>0,
\end{equation}
where $h(t)$ is the average of the mean curvature,
\[h(t) =\frac{\int_{M_t} H dg_t}{\int_{M_t} dg_t}, \]
and $g_t$ denotes the metric on $M_t$. The surface area $|M_t|$ of the hypersurface is known to decrease under the flow (see \cite{MA1}).\\

We are interested in an axially symmetric surface, which encloses the volume $V$, and which has a nonempty boundary contained in a plane $\Pi$ that is perpendicular to the axis of rotation. Motivated by the fact that the stationary solution to the associated Euler Lagrange equation satisfies a Neumann boundary condition, we also assume the surface  to meet that plane $\Pi$ at right angles along its boundary. Assuming the surface to be smooth, it will also intersect orthogonally the axis of rotation.

We consider the case where the surface will not pinch-off along the axis of rotation during the flow, having only one intersection with that axis at the point that is the furthest from the supporting plane $\Pi$, and prove that the surface converges to a halfsphere.

The methods we use also apply in the case of an axially symmetric surface without boundary having a similar lower height bound, and in that case we prove in Section 8 that the flow converges to a sphere. 

The results in this paper make use of the axial symmetry, and no additional conditions on the curvature of the surface are assumed. Converge to spheres has been previously proved for the volume flow by Huisken in \cite{GH2}, for compact, uniformly convex initial surfaces; and  in \cite{HL1}, Li assumes bounds on the traceless second fundamental form.

Our results can be seen as complementing the work of the first author in \cite{MA1,MA2}, and the second author's work on her PhD dissertation: In the case of the surface behaving like a ``bridge" between two parallel surfaces, if one were able to flow through singularities, the axially symmetric volume-preserving flow would converge to a number of spheres and (possibly) two hemispheres on the parallel planes, like beads strung along the axis of rotation.

\section{Notation, definitions and assumptions}
In the case of the surface $M_t$ intersecting the obstacle $\Pi$, we will at different stages divide it into two parts as in \cite{AAG}: one adjacent to the plane and the remainder that contains the (only) intersection with the axis of rotation. 

Let $\Pi = \{ (x_1, \ldots, x_{n+1}) \in \bigR^{n+1} : x_1 =0 \}$ and $M_t$ be contained in the right halfspace, $M_t \subset \{ x_1 > 0 \}$. We use $R_t$ as the generic notation for the part of the surface closest to the plane, and $C_t$ for the rest - the cap that intersects the axis of rotation - , and we will introduce various superscripts depending on the situation that will be made clear in the text. 

We denote by $P(t) = (d(t), 0)$ the ``pole": the point of intersection of $M_t$ with the axis of rotation. We assume that there are no singularities developing, so that $P(t)$ is the only point of intersection of $M_t$ with the axis of rotation for all time. We are interested in those solutions where the generating curve of the initial hypersurface is smooth and can be written
as a graph over the $x_1$ axis except at the pole.

We use the notation
\[ \rho_t: [0, d(t)] \rightarrow \mathbb{R} \]
for the radius function of the surface of revolution. 

Let $\mathbf{i}_1, \cdots , \mathbf{i}_{n+1} $ be the standard basis in $\mathbb{R}^{n+1}$ and $\mathbf{i}_1$ be the direction of the axis of rotation. We denote the quantities associated with the cap with a  tilde $\, \tilde{ }$ , and  in this context we work with the vertical graph equation.

\smallskip

Furthermore we define the following quantities on $M_t$: 

\noindent
Let $\mathbf{\omega} = \frac{ \hat{\mathbf{x} } }{| \hat{\mathbf{x} } |}  \in \bigR^{n+1} \,$,  $\hat{\mathbf{x} } = (0, x_2, \ldots, x_{n+1})$, denote the outer unit normal to the cylinder intersecting $M_t$ at the point $\mathbf{x}(p,t)$. We call $u = \langle \mathbf{x}, \omega \rangle$ the {\it height function} of $M_t$, and set $v = \langle \nu, \omega \rangle^{-1}$. Note that $v$ corresponds to $\sqrt{1+\dot\rho^2}$, and will be used to obtain gradient estimates.

\noindent
The respective quantities on the cap $C_t$ are the height measured from the plane $\Pi$, $\tilde{u} = \langle \mathbf{x}, \mathbf{i}_1 \rangle$, and $\tilde{v} = \langle \nu, \mathbf{i}_1\rangle^{-1}$. 

\noindent
We cut the hypersurface in two regions using the plane $L_\alpha(t)$, which is parallel to $\Pi$ where $\langle \nu, \mathbf{i}_1 \rangle \mid_{L_\alpha(t) \cap M_t} = 1/\alpha \, , $ with $\alpha$ being a constant.  We define the {\it cap}, determined by the inclination angle, as the connected component of $M_t$ containing the pole $P$
$$ C_t^{\alpha} = \left\{ \mathbf{x}(p,t) \in M_t : \frac{1}{\alpha} < \langle \nu, \mathbf{i}_1 \rangle \leq 1 \right\} \, , $$
and we call $R_t^{\alpha} = M_t \backslash C_t^{\alpha}$ the {\it cylindrical part} of the surface. Note that   
$L_\alpha(t)$ is chosen such that the specific inclination angle is achieved nowhere else between that plane and the pole $P(t)$. As long as the flow is smooth, $C_t^{\alpha}$ is by definition a graph over the $x_1$ axis except at the pole.

\begin{assumption}\label{A1}
 We assume for any $\alpha  > 1 $ \, there exists a constant $c(\alpha) > 0$ depending only on $\alpha$ such that $u\mid_{R_t^{\alpha}} > c(\alpha) \, , $ i.e we assume a lower height bound in $R_t^{\alpha}\, ,$ independent of time, dependent on $\alpha \, .$  
\end{assumption}

\noindent
Thus $P(t)$ is the only point of intersection of $M_t$ with the axis of rotation for all time. The assumption prevents singularities developing on the axis of rotation. 
\smallskip

For an axially symmetric surface the mean curvature is given by
\[ H =  -\frac{\ddot{\rho}}{ (1+\dot{\rho}^2)^{\frac{3}{2}}}  + \frac{n-1}{\rho (1+ \dot\rho^2)^{\frac{1}{2}}} \, ,\]
while the principal curvatures are $k = -\frac{\ddot{\rho}}{(1+\dot{\rho}^2)^{\frac{3}{2}} }$ and $p =\frac{1}{\rho \sqrt{1 + \dot{\rho}^2 }} $.
\noindent
We also introduce another quantity $q= \left\langle \nu, \mathbf{i}_1 \right\rangle u^{-1}$ such that in particular $p^2 + q^2 = u^{-2}$.

\section{Height Estimates}
In this section we prove that $M_t$ satisfies uniform height bounds: both, the height function $u$ as defined above, and also the height when measured as distance from the obstacle $\Pi$, i.e. $\tilde u$, are bounded.
\begin{lemma} \label{HeightEstimate1}
The evolving surfaces $M_t$ satisfy the uniform height bound 
$$u <R = \left(\frac{|M_0|}{\omega_n} \right)^{\frac{1}{n}} \, . $$
\end{lemma}

\begin{proof} We follow a method as in \cite{MA1} in getting bounds for $u$. Assume there exists an $R$ such that $u_{M_t} \geq R$ at some given time $t$. Since the surface area is decreasing under the flow, and by comparing to the projection of the surface onto the plane, we have
\[ |M_0| \geq |M_t| > \omega_n R^n, \]
\noindent
where $\omega_n$ is the volume of the $n$ dimensional unit ball. Therefore
\[ R > \left(\frac{|M_0|}{\omega_n} \right)^{\frac{1}{n}} \]
would contradict the fact that the evolution decreases the surface area. 
\end{proof}

\smallskip

\begin{lemma} \label{HeightEstimate2}
There is a constant $l$ such that the evolving surfaces $M_t$ satisfy the height bound 
$$\tilde{u}   \leq l   \, ,$$ 
that is the distance from the plane $\Pi$ is uniformly bounded.
\end{lemma}

\begin{figure}[!h]
   \centering
   \scalebox{0.5}{\pagestyle{empty}
\newrgbcolor{xdxdff}{0.49 0.49 1}
\psset{xunit=1.0cm,yunit=1.0cm,dotstyle=o,dotsize=3pt 0,linewidth=0.8pt,arrowsize=3pt 2,arrowinset=0.25}
\begin{pspicture*}(-4,-3.58)(4.04,3.72)
\psline{->}(-3,-3)(-3,3)
\psline{->}(-3,0)(3,0)
\pscurve(-3,2.26)(-2.75,2.26)(-1.18,1.2)(1.02,2.26)(1.66, 2)(2.2,0)(1.66, -2)(1.02,-2.26)(-1.18,-1.2)(-2.75,-2.26)(-3,-2.26)
\psdots[dotstyle=*](1.66, 2)
\psline[linestyle=dashed,dash=5pt 5pt](1.66, 1)(-3,1)
\psline[linestyle=dashed,dash=5pt 5pt](1.66, -1)(-3,-1)
\psline{->}(1.66,2)(2.5,2)
\psline{->}(1.66,2)(2.3,2.66)
\psline(1.66,2)(1.66,-2)
%\uput[22](1.94, 2.12){\small $\frac{\pi}{4}$}
\uput[22](1.94, 2.12){\small $\theta$}
\uput[0](2.5,2){$i_1$}
\uput[45](2.3,2.66){$\nu$}
\uput[-135](1.66, 0){\small $ L_{\alpha}(t)$}
\psline{->}(0, 0)(0, 1)
\uput[0](0, 0.5){$c(\alpha)$}
\uput[0](3,0){$x_1$}
\uput[90](-3,3){$u$}
\uput[-45](2.2,0){$d(t)$}
\psdots[dotstyle=*](2.2,0)
\psdots[dotstyle=*](1.66, 0)
\end{pspicture*}}
		\caption{The cylinder of radius $c(\alpha)$.}
		\label{fig: Height Estimate}
\end{figure}
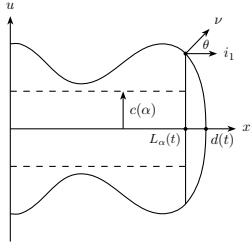

\begin{proof} 
Here $\alpha = \frac{1}{\cos \theta}$, and 
$$ C_t^{\alpha} = \left \{ \mathbf{x}(p,t) \in M_t : \frac{1}{\alpha} < \left \langle \nu, \mathbf{i}_1 \right \rangle \leq 1 \right\} \, , $$
and $R_t^{\alpha} = M_t \backslash C_t^{\alpha}$. From the Assumption \ref{A1}, we know that $u > c(\alpha) $ in $R_t^{\alpha}\, .$ As $u\mid_{\partial C_t^{\alpha}} \leq R  $ and $|{\rho'}| \geq \tan\left(\frac{\pi}{2} - \theta\right) =\frac{1}{\sqrt{\alpha^2 - 1}} $ in $C_t^{\alpha}$ we have 
$$ d(t) - \tilde{u}\mid_{\partial C_t^{\alpha}} \leq R \tan\theta = R \sqrt{\alpha^2 -1} \, . $$
\noindent
Assume there exists a length $l_1$ such that $\tilde{u}\mid_{R_t^{\alpha}} > l_1  $. Then 
\[ |M_0| \geq |M_t| >n \omega_n c^{n-1}(\alpha) l_1 \, , \]
\noindent
where now we compared $\vert M_t \vert $ to the surface area of an $n$ dimensional cylinder of radius $c(\alpha)$ and length $l_1$. 
Having $l_1 > \frac{|M_0|}{n\omega_n c^{n-1}(\alpha)}$ would contradict the fact that the evolution decreases the surface area. Therefore
\[ \tilde{u} <  \frac{|M_0|}{n\omega_n c^{n-1}(\alpha)} + R  \sqrt{\alpha^2 -1}  = \colon l \, .\]
\end{proof}

%--------------------------------------------------------------------------------------------------------
\begin{comment}
\begin{proof} 
Here $\alpha = \sqrt 2$, and 
$$ C_t^{\sqrt2} = \{ \mathbf{x}(p,t) \in M_t : \frac{1}{\sqrt 2} < \langle \nu, \mathbf{i}_1 \rangle \leq 1\} \, , $$
and $R_t^{\sqrt 2} = M_t \backslash C_t^{\sqrt 2}$. From the Assumption \ref{A1}, we know that $u > c(\frac{1}{\sqrt{2}}) =: c_0 $ in $R_t^{\sqrt 2}\, .$ As $u\mid_{\partial C_t^{\sqrt 2}} \leq R$ and $|\dot{\rho}| \geq 1 $ in $C_t^{\sqrt 2}$ we have 
$$ d(t) - \tilde{u}\mid_{\partial C_t^{\sqrt2}} \leq R \, . $$

Assume there exists a length $l_1$ such that $\tilde{u}\mid_{R_t^{\sqrt 2}} > l_1  $. Then 
\[ |M_0| \geq |M_t| >n \omega_n c_0^{n-1} l_1 \, , \]
where now we compared to the surface area of an $n$ dimensional cylinder of radius $c_0$ and length $l_1$. 
If $l_1 > \frac{|M_0|}{n\omega_n c_0^{n-1}}$ this would contradict the fact that the evolution decreases the surface area. Therefore
\[ \tilde{u} <  \frac{|M_0|}{n\omega_n c_0^{n-1}} + R \, .\]
\end{proof}
\end{comment}
%---------------------------------------------------------------------------------------------------------------

\noindent
Next we show that the length of the generating curve is bounded.
 
\begin{lemma}\label{Lemma_Length_of_the_generating_curve_is_bounded}
Assume $M_t$ to be a smooth, rotationally symmetric hypersurface, with a radius function $\rho(x_1, t) > 0 $ for $x_1 \in [0, d(t))$. Then there exists a constant $c_*$, such that
$$ \int_{0}^{d(t)} \sqrt{1 + \rho'^2} \dx_1 \leq c_* \, , $$
independent of time.
\end{lemma}
\begin{proof}
Let us divide $M_t$ into $R_t^{\alpha}$ and $C_t^{\alpha}$ for any $\alpha >1$. As the surface area is decreasing under the flow $$ \vert M_t \vert \leq \vert M_0 \vert \, , $$
$$ 2\pi \int_{0}^{d(t)} \rho^{n-1}\sqrt{1 + \rho'^2} \dx_1 \leq \vert M_0 \vert \, , $$
$$ 2\pi \int_{0}^{L_{\alpha}(t)} \rho^{n-1}\sqrt{1 + \rho'^2} \dx_1 \leq 2\pi \int_{0}^{d(t)} \rho^{n-1}\sqrt{1 + \rho'^2} \dx_1  \leq \vert M_0 \vert \, . $$
\noindent
From the Assumption \ref{A1} 
$$ 2 \pi c^{n-1}(\alpha) \int_{0}^{L_{\alpha}(t)}\sqrt{1 + \rho'^2} \dx_1 \leq \vert M_0 \vert \, , $$
$$  \int_{0}^{L_{\alpha}(t)}\sqrt{1 + \rho'^2} \dx_1 \leq \frac{\vert M_0 \vert} {2\pi c^{n-1}(\alpha)} \, . $$
\noindent
We can estimate the length of the generating curve of the cap $C_t^{\alpha}$ by $l +R$. Therefore
$$ \int_{0}^{d(t)} \sqrt{1 + \rho'^2} \dx_1  \leq \frac{\vert M_0 \vert} {2\pi c^{n-1}(\alpha)} + l + R =\colon c_* \, . $$
\end{proof}
\noindent
We now derive an   \emph{a priori} estimate for $h(t)$ for any solution of the graphical equation. 

\section{Estimates on $h$}
\begin{lemma}\label{estimate on $h$}
Assume $M_t$ to be a smooth, rotationally symmetric hypersurface, with a radius function $\rho(x_1, t) > 0 $ for $x_1 \in [0, d(t))$.  Then there is a constant $c_1$ such that $0 \leq h(t) \leq c_1$ throughout the flow.
\end{lemma}

\begin{proof}
Following \cite{MA2} we parametrize $ M_t $ by its radius function $\rho \in C^\infty([0, d(t)))$, then clearly
\[ H =  -\frac{\ddot{\rho}}{ (1+\dot{\rho}^2)^{\frac{3}{2}}}  + \frac{n-1}{\rho (1+ \dot\rho^2)^{\frac{1}{2}}} \, .\]

\noindent
From Lemma \ref{Lemma_Length_of_the_generating_curve_is_bounded}, we know that $\int_{0}^{d(t)} \sqrt{1+ \dot{\rho}^2} dx_1 \leq c_* $. Our proof follows the ideas of \cite{MA1} the difference being the boundary term when integrating by parts. For the sake of completeness we will include it here. For the second term of
\[ h(t) = \frac{1}{|M_t|} \int_{M_t}(k + (n-1)p ) dg_t, \hspace{2mm} t\in [0, T) \, , \]
we have
\[ 0 \leq \frac{n-1}{|M_t|} \int_0^{d(t)} \rho ^{n-2} (x_1, t) dx_1 \leq \frac{(n-1)R^{n-2}l}{|M_t|} \, , \] 
since $\rho \leq R$ and $d(t) \leq l$ by Lemma \ref{HeightEstimate1} and \ref{HeightEstimate2}.

\noindent
For the first term note that $\frac{\ddot{\rho}}{ (1+\dot{\rho}^2)} = \frac{d}{dx_1}(\arctan \dot\rho)$. Therefore 
\begin{align}\label{hBoundaryValues}
\int_{M_t} k dg_t &= - \int_0^{d(t)} \frac{d}{dx_1}(\arctan \dot\rho) \rho^{n-1} dx_1  \notag \\
& = (\arctan \dot\rho) \rho^{n-1}\mid_{x_1 = 0 } - (\arctan \dot\rho) \rho^{n-1}\mid_{x_1 = d(t) } + (n-1) \int_0^{d(t)} (\arctan \dot\rho) \dot\rho \rho^{n-2} dx_1 \\
& = (n-1) \int_0^{d(t)} (\arctan \dot\rho) \dot\rho \rho^{n-2} dx_1 \, ,\notag
\end{align}

\noindent
as $\arctan \dot \rho = 0$ when $x_1 = 0 $, and $\rho(d(t)) = 0 $ at the pole. 
As $0\leq (\arctan \dot\rho) \dot\rho \leq \frac{\pi}{2}|\dot\rho | \leq \frac{\pi}{2} \sqrt{1+\dot\rho^2}$ we obtain
\begin{align*}
0 \leq \frac{1}{|M_t|} \int_{M_t} k dg_t & \leq \frac{(n-1)}{|M_t|} \frac{\pi}{2} \int_0^{d(t)} \sqrt{1+\dot\rho^2} \rho^{n-2} dx_1 \\
& \leq \frac{(n-1)R^{n-2}}{|M_t|} \frac{\pi}{2}\int_0^{d(t)} \sqrt{1+\dot\rho^2} dx_1 \\
&\leq \frac{(n-1)c_*R^{n-2}}{|M_t|} \frac{\pi}{2}  \, ,
\end{align*}
where we have used Lemma \ref{Lemma_Length_of_the_generating_curve_is_bounded}. 

\noindent
From the isoperimetric inequality and the fact that the flow decreases surface area we know that 
\[ V ^{\frac{n}{n+1}} < c|M_t| \leq c|M_0|. \]
Hence combining these arguments we conclude
\[ 0\leq \frac{\int H dg}{\int dg} \leq c_1. \]

\end{proof}

\section{Evolution equations and gradient estimates}

The maximum principle for non-cylindrical or time dependent domains is discussed in \cite{LS1}. We use that version of the maximum principle in this paper. 

\begin{lemma}\label{Lemma_EvEq}We have the following evolution equations:
\begin{itemize}
\item[(i)] $\left( \frac{d}{dt} - \Delta \right) u= \frac{h}{v} - \frac{n-1}{u}\, ;  $
\item[(ii)] $\left( \frac{d}{dt} - \Delta \right) \tilde{u}= \frac{h}{\tilde{v}}\, ;  $
\item[(iii)] $\left( \frac{d}{dt} - \Delta \right) v= -|A|^2v + (n-1)\frac{v}{u^2} - \frac{2}{v}|\nabla v|^2\, ;  $
\item[(iv)] $\left( \frac{d}{dt} - \Delta \right) \tilde{v}= -|A|^2\tilde{v} - \frac{2}{\tilde{v}}|\nabla \tilde{v}|^2\, ;  $
\item[(v)] $\left( \frac{d}{dt} - \Delta \right) H= (H-h)|A|^2\, ;  $
\item[(vi)] $ \left( \frac{d}{dt} - \Delta \right) |A|^2= -2|\nabla A|^2 + 2 |A|^4 - 2hC\, ;  $
\item[(vii)] $ \left( \frac{d}{dt} - \Delta \right)  p =  |A|^2p + 2q^2(k-p) - hp^2 \, ;  $
\item[(viii)] $\left(  \frac{d}{dt} - \Delta \right) k =  |A|^2k -2(n-1)q^2(k-p) - hk^2 \, ;  $
\end{itemize}
where $C = g^{ij}g^{kl}g^{mn}h_{ik}h_{lm}h_{nj} $, with $ g^{ij}$ denoting the components of the inverse of the first fundamental form, and $h_{ij}$ those of the second fundamental form.
\end{lemma}
\begin{proof}
(i) and (iii) are proved in \cite{MA1}, (v) and (vi) are in \cite{GH2}. 

\smallskip

\noindent
(ii) For $\tilde{u} = \left\langle \mathbf{x}, \mathbf{i}_1 \right\rangle$ we have 
$$\frac{d}{dt} \tilde{u}  = \left\langle \frac{d}{dt}\mathbf{x}, \mathbf{i}_1 \right \rangle 
 = -(H-h)\langle \nu , \mathbf{i}_1 \rangle \, , $$
and
$$ \Delta \tilde{u} = \langle \Delta \mathbf{x}, \mathbf{i}_1 \rangle 
= -H \langle \nu , \mathbf{i}_1 \rangle \, , $$
so that
$$ \left( \frac{d}{dt} - \Delta \right) \tilde{u}  = h \langle \nu , \mathbf{i}_1 \rangle \, . $$

\smallskip

\noindent
(iv) For $\tilde{v} = \langle \nu, \mathbf{i}_1\rangle^{-1}$ we have 

$$ \frac{d}{dt} \tilde{v} = -\tilde{v}^2 \left\langle \frac{d}{dt}\nu, \mathbf{i}_1 \right\rangle = -\tilde{v}^2 \langle \nabla H, \mathbf{i}_1\rangle \, .$$

\noindent
The evolution equation follows from the following well known  identity (see for example \cite{HE1})
\[ \Delta \tilde{v} =  -\tilde{v}^2 \langle \nabla H, \mathbf{i}_1\rangle + \tilde{v}|A|^2 + 2\tilde{v}^{-1} \nabla{\tilde{v}}^2 \, .\]

\noindent 
Following the same way as in \cite{GH3}, for (vii) we have ,

\begin{align*}
\frac{d}{dt} p &= \frac{d}{dt}(u^{-2} -q^{2})^{1/2} \\
	&=\Delta p + p^{-1}|\nabla p|^2 + p^{-1}|\nabla q|^2 -3 p^{-1}u^{-4}|\nabla u|^2 +p^{-1}u^{-4} \\
	& \hspace{2mm} -qp^{-1}\left(|A|^2q + q\left(p^2-q^2 -2kp\right)\right) -hu^{-2} +hq^2 \,.
\end{align*}
Equation (vii) follows from the following relations:
\begin{align}\label{eq:T1.4}
\nabla_i u &= \delta_{i1}qu  \, , \hspace{10mm} \nabla_1 \left\langle \nu, \mathbf{i}_1 \right\rangle =kpu \, ,
	\hspace{10mm}\nabla_i q = \delta_{i1}(q^2 + kp)  \, , \\ 
 \nabla_i p &= \delta_{i1}q(p-k) \, ,\hspace{5mm} |A|^2 = k^2 +(n-1) p^2 \, , \hspace{5mm} u^{-4} = p^4 + 2p^2q^2 + q^4 \,.
\end{align}

\smallskip
\noindent
The evolution equation for $H$ was derived in \cite{GH2} , and (viii) follows from (v), (vii), and the fact that $H =k+(n-1)p$.
\end{proof}

We proceed to obtain gradient estimates in the different parts of the surface: for the cap using the vertical graph equation and part (iv) from above Lemma \ref{Lemma_EvEq}, and for the cylindrical part away from the cap, using the evolution equation (iii) in Lemma \ref{Lemma_EvEq}.

\noindent
The quantities $\tilde{u}$ and $\tilde{v}$ are used on the cap. 

\begin{lemma}\label{gradient estimate}
The gradient estimate
\[\tilde{v} \leq \alpha \]
holds on the cap $C_t^{\alpha} $. In addition there is a constant $c_2(\alpha)$, such that
\[v \leq c_2(\alpha)  \]
for the cylindrical part $R_t^{\alpha}$.
\end{lemma}

\begin{proof}
Note that 
\[ \left( \frac{d}{dt} - \Delta \right) \tilde{v} \leq 0 \, ,\]
so that by the maximum principle  $ \tilde{v} \leq \max (\max_{C_0^{\alpha}} \tilde{v}, \max_{\partial C_t^{\alpha}} \tilde{v})$. By definition in  $C^{\alpha}_t $ we have  $\tilde{v} \leq \alpha $, and this is supported by the evolution equation!

\noindent
From the assumption we know that $u > c(\alpha)$ in $R^{\alpha}_t. $ 
As in (\cite{MA1}, Proposition $4$) we calculate, 
\begin{align*}
\left( \frac{d}{dt} - \Delta \right)u^2 v &= -|A|^2u^2v + (n-1)v + 2uh - 2(n-1)v -2v|\nabla u|^2  \\
						 & \hspace{2mm}	-  \frac{2}{v}\nabla v \nabla (u^2 v) \\
						 & \leq 2hu - (n-1)v \, . 
\end{align*}
If $ v > \frac{2c_1 R}{n-1}$ the right hand side is negative, and proceeding as in \cite{MA1} we conclude $v \leq c_2(\alpha)$ in $R^{\alpha}_t$. It is important to note that on the boundary of $R^{\alpha}_t$, either $v=1$ (along the intersection with $\Pi$), or $v=\frac{\alpha}{\sqrt{ {\alpha}^2 - 1}}$. Thereby, we have bounds for $v$ and $\tilde{v}$ in $R^{\alpha}_t$ and $C^{\alpha}_t$ respectively. 
\end{proof}

\begin{remark}
\noindent
These gradient bounds guarantee that $R^{\alpha}_t$ remains a graph. As $R^{\alpha}_t$ remains a graph for all $\alpha >1$, we see that  $M_t\backslash P(t)$ remains a graph as well. 
\end{remark}

\noindent
 As the height of the graph is bounded, for the minimum $d(t)$  we find a lower bound from
\[ V = \int_0^{d(t)} \omega_n \rho^n(x) dx_1 \leq  \omega_n  R^n  \int_0^{d(t)} dx_1  = \omega_n  R^n  d(t) \,.\]
  
\begin{lemma} \label{HPositive}
 Let $\mathbf{x}_0 (t)$ be a boundary point of $C_t^{\sqrt{2}}$, which without loss of generality we can assume to lie on the generating curve and such that $ \langle \nu(\mathbf{x}_0 (t)), \mathbf{i}_1 \rangle = \frac{1}{\sqrt 2} $ (with some abuse of notation for the corresponding normal $\nu(\mathbf{x}_0(t))$). Then $H(\mathbf{x}_0 (t)) \geq 0$ for $0 \leq t \leq T_{max} \leq \infty $. 
\end{lemma}

\begin{proof}

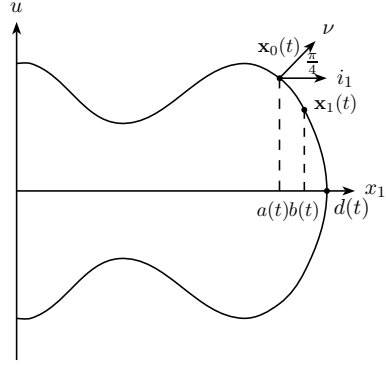
\begin{figure}[!h]
   \centering
   \scalebox{0.75}{\pagestyle{empty}
\newrgbcolor{xdxdff}{0.49 0.49 1}
\psset{xunit=1.0cm,yunit=1.0cm,dotstyle=o,dotsize=3pt 0,linewidth=0.8pt,arrowsize=3pt 2,arrowinset=0.25}
\begin{pspicture*}(-4,-3.58)(4.04,3.72)
\psline{->}(-3,-3)(-3,3)
\psline{->}(-3,0)(3,0)
\pscurve(-3,2.26)(-2.75,2.26)(-1.18,1.2)(1.02,2.26)(1.66, 2)(2.1, 1.44)(2.5,0)(2.1, -1.44)(1.66, -2)(1.02,-2.26)(-1.18,-1.2)(-2.75,-2.26)(-3,-2.26)
\psdots[dotstyle=*](1.66, 2)
\psdots[dotstyle=*](2.1, 1.44)
\psline[linestyle=dashed,dash=5pt 5pt](1.66, 2)(1.66, 0)
\psline[linestyle=dashed,dash=5pt 5pt](2.1, 1.44)(2.1, 0)
\psline{->}(1.66,2)(2.5,2)
\psline{->}(1.66,2)(2.3,2.66)
\uput[22](1.94, 2.12){\small $\frac{\pi}{4}$}
\uput[0](2.5,2){$i_1$}
\uput[45](2.3,2.66){$\nu$}
\uput[-90](1.55, 0){\small $a(t)$}
\uput[90](1.66,2.2){\small $\mathbf{x}_0(t)$}
\uput[15](2.1, 1.44){\small $\mathbf{x}_1(t)$}
\uput[-90](2.1, 0){\small $b(t)$}
\uput[0](3,0){$x_1$}
\uput[90](-3,3){$u$}
\uput[-45](2.5,0){$d(t)$}
\psdots[dotstyle=*](2.5,0)
\end{pspicture*}}
		\caption{$\mathbf{x}_0 (t)$ - the boundary point of $C_t^{\sqrt{2}}$.}
		\label{fig: Boundary of H positive}
\end{figure}

Suppose $H(\mathbf{x}_0 (t)) <0 $, then by continuity there is a connected region $ C_t^{\sqrt{2},H^-} \subset  C_t^{\sqrt{2}}$, with $\mathbf{x}_0(t) \in \partial C_t^{\sqrt{2},H^-} $, which clearly can be chosen to be axially symmetric, and such that  $H \mid_{C_t^{\sqrt{2},H^-} } <  0 $. Let $\mathbf{x}_1 (t )$ denote the other boundary point along the generating curve in $C_t^{\sqrt{2},H^-} \subset C_t^{\sqrt{2}} $, and let $a(t )= \langle \mathbf{x}_0 (t) , \mathbf{i}_1 \rangle $, $b(t) = \langle \mathbf{x}_1 (t) , \mathbf{i}_1 \rangle $ denote the $x_1$ coordinate of $\mathbf{x}_0 (t)$, $\mathbf{x}_1 (t)$, respectively. Then
\begin{equation*}
0 > \int_{C_t^{\sqrt{2},H^-} } H dg = \int_{a(t)}^{b(t)} \left( -\frac{\ddot{\rho}}{1+\dot{\rho}^2} \rho^{n-1} + (n-1)\rho^{n-2} \right)dx_1 \, .
\end{equation*}
The second term being positive, that means that the first is negative, and given the bounds on the radius we find
\begin{equation*}
\int_{a(t)}^{b(t)}  \left(-\frac{\ddot{\rho}}{1+\dot{\rho}^2} \right) dx_1 = \int_{a(t)}^{b(t)} \left(- \frac{d}{dx_1}(\arctan \dot{\rho})  \right) dx_1< 0 \, . 
\end{equation*}
This results in $\arctan \dot{\rho} (a(t))  <   \arctan \dot{\rho} (b(t))$ and by the choice of $a(t)$, 
$-\frac{\pi}{4} <  \arctan \dot{\rho} (b(t)) $. However, this is not possible in $C_t^{\sqrt{2}}$, where 
$ -\frac{\pi}{2} \leq \arctan \dot{\rho} < -\frac{\pi}{4}$, contradicting our assumption and therefore $H(\mathbf{x}_0 (t)) \geq 0 $.
\end{proof}

\section{Curvature estimates}
\begin{proposition} \label{eq:3.1}There is a constant $c_2$ depending only on the initial hypersurface, such that  $\frac{k}{p} < c_2$, independent of time.
\end{proposition}
\begin{proof}
 We calculate from Lemma \ref{Lemma_EvEq}
\[ \frac{d}{dt} \left( \frac{k}{p} \right) = \Delta \frac{k}{p} +
	\frac{2}{p}\nabla_ip\nabla_i\left(\frac{k}{p}\right) +
	2\frac{q^2}{p^2}\left(p - k\right)\left((n-1)p + k\right) +
	\frac{hk}{p}\left(p-k\right) \, .\]
If $\frac{k}{p} \geq 1$ then $\frac{hk}{p}\left(p-k\right)<0$. This implies that
\begin{equation}
\frac{k}{p} \leq \max \left(1, \max_{M_0} \frac{k}{p} \right)  \, . 
\end{equation}
Note that for this consideration, the smooth function $k/p$ is defined over the whole surface, and in view of the orthogonality on the boundary, via a reflection argument there are no boundary data involved.
\end{proof}

\begin{proposition}\label{boundsforcurvature}
There exists a constant $c_3$ such that
$$ |A|^2 \leq c_3 . $$ 
\end{proposition}
\begin{proof}
We proceed as in  \cite{KE2} and  \cite{MA1} and calculate the evolution equation for the product $g = |A|^2 \varphi(v^2)$ in $R_t^{\sqrt{2}}$ , where $\varphi(r) = \frac{r}{\lambda - \mu r}$, with some constants $\lambda, \mu >0 $ and $v = \langle \nu, \omega \rangle ^{-1}$. From the evolution equation of $g$ we find the inequality
\[\left ( \frac{d}{dt}- \Delta \right) g \leq -2\mu g^2 -2\lambda \varphi v^{-3}\nabla v \cdot \nabla g - \frac{2\lambda \mu}{(\lambda-\mu v^2)^2}|\nabla v|^2 g -2hC \varphi(v^2) 
+ \frac{2(n-1)}{u^2}v^2 \varphi'|A|^2  \, .\]
We estimate the second last term as in \cite{MA1} using Young's inequality and obtain
\begin{align*}
-2hC\varphi(v^2) &\leq 2h|A|^3\varphi(v^2) \\
&\leq \frac{3}{2}|A|^4\varphi(v^2) + \frac{1}{2}h^4\varphi^{-2}(v^2) \\
&= \frac{3}{2}g^2 + \frac{1}{2}h^4\varphi^{-2}(v^2) \, .
\end{align*}
We choose $\mu > \frac{3}{4}$ and $\lambda > \mu \max v^2 $. As $\varphi' v^2= \frac{\lambda}{(\lambda-\mu v^2)^2}\varphi $ we have
\[ \frac{2(n-1)}{u^2}v^2 \varphi'|A|^2  = \frac{2(n-1)\lambda}{u^2(\lambda - \mu v^2)} g \, .\]
As $u >c\left(\frac{1}{\sqrt{2}}\right) = c_0$ in $R_t^{\sqrt{2}}$ we get 
\[ \frac{2(n-1)\lambda}{u^2(\lambda - \mu v^2)} g  \leq c_4 g \, . \] 
Therefore we have
\begin{align*}
\left ( \frac{d}{dt}- \Delta \right) g &\leq -c_5 g^2  + c_{6}g - c_{7}\nabla v \cdot \nabla g + c_{8}(h, \max v)  \\
& \leq -c_5 \left(g- \frac{c_{6}}{2c_5}\right)^2 - c_{7}\nabla v \cdot \nabla g  +  c_{9} \, .
\end{align*}
When $g > \frac{c_{6}}{2c_5} + \sqrt{\frac{c_{9}}{c_5}}$ we get the right hand side to be negative. On $\partial R_t^{\sqrt{2}}$ we have $H = k + (n-1)p \geq 0 $ by Lemma \ref{HPositive} and also as $\frac{k}{p }< c_2$ we get $\frac{|k|}{p} < c$ on the boundary point. Here we have

\[ |A|^2 = k^2 + (n-1)p^2 \leq (c^2 + n -1) p^2 \leq C \rho^{-2} \leq C c_0^{-2}  \, . \]
From the maximum principle 
\[ g \leq \max \left(\max_{R_0^{\sqrt{2}}} g,  \max_{\partial R_t^{\sqrt{2}}} |A|^2 \varphi(v^2) \right) \, .\]
As $\varphi(v^2)$ is bounded as $v$ is bounded we have a bound for $g$ in $R_t^{\sqrt{2}}$. 
\noindent
If we calculate the evolution equation for $\tilde{g} = |A|^2\varphi(\tilde{v}^2)$ on $C_t^{\sqrt{2}}$ then we get the same evolution equation without the last term on the righthand side. Thereby we get a bound for $\tilde{g}$ in the same way as above.  
\end{proof}

\begin{proposition} For each $m \geq 1$ there is $C_m$ such that
\[ |\nabla^m A|^2  \leq  C_m \, ,\]
uniformly on $M_t$ for $0 \leq t \leq T_{\max} \leq \infty $.
\end{proposition}
\begin{proof}
Having obtained uniform bounds on $|A|^2$ and $h$ the proof is a repetition of that of Theorem 4.1 in \cite{GH2}.
\end{proof}

Thus we have long-time existence for the flow. 

\begin{corollary}
\[ T_{\max} = \infty \, .\]
\end{corollary}

\section{Convergence to surfaces of constant mean curvature}
This is as in \cite{MA1}: Having long-time existence, Proposition 8 of \cite{MA1} gives convergence to a constant mean curvature surface, which in our case is axially symmetric. By the classification of the Delaunay surfaces \cite{DEL} it has to be a half-sphere.

\section{Other convergence results}
Using the same estimates with very few changes one can show that a compact,  axially symmetric surface  without boundary,  which encloses a volume $V$  and intersects the axis only at two endpoints throughout the flow, will converge to a sphere. We will only explain the parts that are different from the previous result. 

\subsection{Height Estimates}
The height estimates \ref{HeightEstimate1} and \ref{HeightEstimate2} change as follows. 

\begin{lemma} {The height function $u$ satisfies $u <R = \left(\frac{|M_0|}{2\omega_n} \right)^{\frac{1}{n}} $.}
\end{lemma}

\begin{proof}

Assume there exists an $R$ such that $u_{M_t} \geq R$ at some given time $t$. Take a plane perpendicular to the $x_1$-axis and intersecting the surface. This plane divides the surface into two parts, and by projecting both parts onto the plane we find   
\[ |M_0| \geq |M_t| > 2\omega_n R^n \, . \]
Taking  \[ R > \left(\frac{|M_0|}{2\omega_n} \right)^{\frac{1}{n}} \]
would contradict the fact that the evolution decreases the surface area. 
\end{proof}

\begin{lemma} { Let $e(t)$ and $d(t)$ denote the two tips of the surface on the left hand and right hand side respectively. Then $d(t) - e(t) <l = \frac{|M_0|}{n\omega_n c_0^{n-1}} + 2 R  $.}
\end{lemma}

\begin{proof}
As in Lemma \ref{HeightEstimate2} let $\alpha = \frac{1}{\cos \theta}$.
From the Assumption \ref{A1} we know that $u > c(\alpha)  $ in $R_t^{\alpha}\, .$ As $u\mid_{\partial C_t^{\alpha, i}} \leq R $ and $|\dot{\rho}| \geq \tan\left(\frac{\pi}{2} - \theta\right) $ in $C_t^{\alpha,i}$ for $i =\{1,2 \} \, ,$ we have 
$$ d(t) - \tilde{u}\mid_{\partial C_t^{\alpha,1}} \leq R \tan\theta = R \sqrt{\alpha^2 -1}  \, , \, \text{and} $$
$$\tilde{u}\mid_{\partial C_t^{\alpha,2}} -  e(t) \leq R \tan\theta = R \sqrt{\alpha^2 -1}  \, . $$
Assume there exists a length $l_1$ such that $\tilde{u}\mid{R_t^{\alpha}} > l_1 $. Then by the previous argument again we can say
\[ |M_0| \geq |M_t| >n \omega_n c^{n-1}(\alpha) l_1 \, , \]
where now we compared $|M_t|$ to the surface area of an $n$ dimensional cylinder of radius $c(\alpha)$ and length $l_1$. 
If $l_1 > \frac{|M_0|}{n\omega_n c^{n-1}(\alpha)}$ this would contradict the fact that the evolution decreases the surface area. Therefore
\[ \tilde{u} <  \frac{|M_0|}{n\omega_n c_0^{n-1}} + 2R \sqrt{\alpha^2 -1} \, .\]
\end{proof}

\begin{lemma}\label{Lemma_Length_of_the_generating_curve_is_bounded_2}
Assume $M_t$ to be a smooth, rotationally symmetric hypersurface, with a radius function $\rho(x_1, t) > 0 $ for $x_1 \in (e(t), d(t))$. Then there exists a constant $c_*$, such that
$$ \int_{0}^{d(t)} \sqrt{1 + \rho'^2} \dx_1 \leq c_* \, , $$
independent of time.
\end{lemma}
\begin{proof}
The proof is the same as in Lemma \ref{Lemma_Length_of_the_generating_curve_is_bounded} after taking into account the two caps on either side. Here we would have
$$ \int_{e(t)}^{d(t)} \sqrt{1 + \rho'^2} \dx_1  \leq \frac{\vert M_0 \vert} {2\pi c^{n-1}(\alpha)} + 2l + 2R =\colon c_* \, .
$$
\end{proof}

\begin{lemma}\label{Estimates on h} {\bf (Estimates on $h$)}
Assume $M_t$ to be a smooth, rotationally symmetric hypersurface, with a radius function $\rho(x_1, t) > 0 $ for $x_1 \in (e(t), d(t))$.  Then there is a constant $c_1$ such that $0 \leq h(t) \leq c_1$ throughout the flow.
\end{lemma}
\begin{proof}
The only change to the proof of Lemma \ref{estimate on $h$} would be in the boundary values \ref{hBoundaryValues}. Here the new boundary values would be
\[(\arctan \dot\rho) \rho^{n-1}\mid_{x_1 = a(t) } - (\arctan \dot\rho) \rho^{n-1}\mid_{x_1 = b(t) } \]
As $\rho(a(t)) = \rho(b(t)) = 0 $, the boundary terms dissapear and we get the same estimate for $h$. 
\end{proof}

\begin{lemma}\label{Gradient Estimates}{\bf(Gradient estimates)}
The gradient estimate
\[\vert\tilde{v}\vert \leq \alpha \]
holds on the caps $C_t^{\alpha,i} $, $i=1,2$. In addition there is a constant $c$, such that
\[ v \leq c  \]
for the cylindrical part $R_t^{\alpha}$.

\end{lemma}
\begin{proof}
The gradient estimates are as in Lemma \ref{gradient estimate}, but in this setting instead of one cap $C^{\alpha}_t$ we have two caps on either side, and the same estimate would suffice for both caps. 
\end{proof}

Concluding this section, we remark that  $H \geq 0$ at points where the caps $C_t^{\sqrt2,i} $, $i=1,2$, meet the cylindrical part $R_t^{\sqrt2}$ of the surface. The proof is using the same arguments as the one for Lemma \ref{HPositive} after the appropriate adjustments of the sign of $\arctan \dot \rho$ for cap on the left of the surface. The results on {\bf curvature estimates} and the {\bf convergence} to a limiting surface of constant mean curvature follow along the same lines as previously proved. In this case the limit surface is a sphere.

\end{document}